\documentclass[11pt]{amsart}

\usepackage{amsfonts,amsmath,amssymb,amsthm,mathrsfs,url,xcolor,latexsym,rawfonts,graphicx}
\usepackage{hyperref}
\usepackage{enumerate} 
\usepackage[]{caption2}
\usepackage{accents}

\theoremstyle{plain}
  \newtheorem{theorem}{Theorem}[section]
  \newtheorem{lemma}{Lemma}[section]
  \newtheorem{proposition}{Proposition}[section]
  
  \newtheorem{corollary}{Corollary}[section]
  \newtheorem{definition}{Definition}[section]
  \newtheorem{remark}{Remark}[section]

\newcommand{\U}{\overline{U}}

   \newcommand{\beqn}{\begin{eqnarray}}
   \newcommand{\eeqn}{\end{eqnarray}}
   \newcommand{\beqs}{\begin{eqnarray*}}
   \newcommand{\eeqs}{\end{eqnarray*}}
   \newcommand{\ban}{\begin{eqnarray*}}
   \newcommand{\nan}{\end{eqnarray*}}
   \newcommand{\beq}{\begin{equation}}
   \newcommand{\eeq}{\end{equation}}
   
   \newcommand{\p}{\partial}
\newcommand{\eps}{\varepsilon}

\newcommand{\Om}{\Omega}

\newcommand{\bom}{{\overline\Om}}

\renewcommand{\det}{\mbox{det}}

\newcommand{\R}{\mathbb{R}}

\numberwithin{equation}{section}

  \numberwithin{equation}{section}
  \numberwithin{figure}{section}

\begin{document}

\title[Global regularity in the Monge-Amp\`ere obstacle problem]
{\textbf{Global regularity in the Monge-Amp\`ere obstacle problem}}

\author{Shibing Chen}
\address{Shibing Chen, School of Mathematical Sciences,
University of Science and Technology of China,
Hefei, 230026, P.R. China.}
\email{chenshib@ustc.edu.cn}

\author[J. Liu]{Jiakun Liu}
\address
	{Jiakun Liu, School of Mathematics and Applied Statistics,
	University of Wollongong,
	Wollongong, NSW 2522, AUSTRALIA}
\email{jiakunl@uow.edu.au}

\author[X. Wang]{Xianduo Wang}
\address{Xianduo Wang, Department of Mathematical Sciences, Tsinghua University, Beijing 100084, China}
\email{xd-wang18@mails.tsinghua.edu.cn}



%    \thanks will become a 1st page footnote.
\thanks{Research of Chen was supported by National Key R\&D program of China 2022YFA1005400, 2020YFA0713100, National Science Fund for Distinguished Young Scholars (No. 12225111), and NSFC No. 12141105.
Research of Liu was supported by ARC DP200101084, DP230100499 and FT220100368.}

%\subjclass[2000]{Primary 53C44, 35K55}

\date{\today}

\dedicatory{}

\begin{abstract} 
In this paper, we establish the global $W^{2,p}$ estimate for the Monge-Ampère obstacle problem: $(Du)_{\sharp}f\chi{_{\{u>\frac{1}{2}|x|^2\}}}=g$, where $f$ and $g$ are positive continuous functions supported in disjoint bounded $C^2$ uniformly convex domains $\overline{\Omega}$ and $\overline{\Omega^*}$, respectively. Furthermore, we assume that $\int_{\Omega}f\geq \int_{\Omega^*}g$. The main result shows that $Du:\overline U\rightarrow\overline{\Omega^*}$, where $ U=\{u>\frac{1}{2}|x|^2\}$, is a $W^{1, p}$ diffeomorphism for any $p\in(1,\infty)$. Previously, it was only known to be a continuous homeomorphism according to Caffarelli and McCann \cite{CM}. 
It is worth noting that our result is sharp, as we can construct examples showing that even with the additional assumption of smooth densities, the optimal map $Du$ is not Lipschitz.
This obstacle problem arises naturally in optimal partial transportation.

\end{abstract}

\maketitle

\baselineskip=16.4pt
\parskip=3pt

\section{introduction}
In this paper we study the convex solution to the Monge-Amp\`ere obstacle problem 
\begin{equation}\label{mao}
(Du)_{\sharp}f\chi_{_{\{u>\frac{1}{2}|x|^2\}}}=g,
\end{equation}
where $0\leq f, g\in L^1(\mathbb{R}^n)$ are supported in the closures of bounded convex domains
$\Omega, \Omega^*$, respectively. We also assume that $\overline{\Omega}$ and $\overline{\Omega^*}$ are disjoint, and that $\int_{\Omega}f\geq \int_{\Omega^*}g.$
The equation \eqref{mao} is understood in Brenier's sense, namely, a convex function $u$ is a solution of \eqref{mao} if and only if for each bounded continuous function $\varphi:\mathbb{R}^n\rightarrow \mathbb{R},$ we have 
$$\int_{_{\{u>\frac{1}{2}|x|^2\}}} \varphi(Du(x))f(x)dx=\int_{\mathbb{R}^n}\varphi(y)g(y)dy.$$

This obstacle problem arises in the optimal partial transport problem.
%In this paper we study the optimal partial transport problem from $(\Omega, f)$ to
%$(\Omega^*, g)$ 
%In this paper we study how the free boundary intersects with the fixed boundary in the optimal partial transport problem. In particular, we show that the intersection is always in a nice way,  namely, we rule out the possibility of nontransverse intersection, which gives an answer to a question raised in \cite[Page 706]{CM}. Using this property, we show that the free boundary is   $C^{1, 1-\epsilon}$ regular up to its intersection with the fixed boundary. 
%Let $\Omega$ and $\Omega^*$ be two disjoint, bounded, convex domains in the Euclidean space $\R^n$.
%Let $f$ and $g$ be the densities in $\Omega$ and $\Omega^*$,  respectively. 
%Let $m$ be a positive constant satisfying 
%\begin{equation}\label{mass1}
%m\leq \min\Big\{\int_\Omega f, \,\, \int_{\Omega^*} g\Big\}.
%\end{equation}
A non-negative, finite Borel measure $\gamma$ on $\mathbb{R}^n\times \mathbb{R}^n$ is called 
a transport plan (with mass $m:=\int_{\Omega^*}g$) from the distribution $(\Omega, f)$ to the distribution $(\Omega^*, g)$,
if $\gamma(\R^n\times\R^n)  =m$ and 
\beq\label{tp}
{\begin{split}
       \gamma(A\times \mathbb{R}^n) \leq \int_{A\cap\Om} f(x)\,dx, \ \ \ 
       \gamma(\mathbb{R}^n \times A) \leq \int_{A\cap\Om^*} g(y)\,dy
       \end{split}}
\eeq
for any Borel set $A\subset\R^n$.
A transport plan $\gamma$ is {\it optimal} if it minimises the cost functional 
\begin{equation} \label{mini}
 \int_{\mathbb{R}^n \times \mathbb{R}^n} |x-y|^2\, d\gamma(x,y) 
\end{equation}    
among all transport plans. 
In \cite{CM}, Caffarelli and McCann have shown the existence and uniqueness of solutions to both
\eqref{mao} and \eqref{mini}, moreover, the optimal plan of \eqref{mini} can be characterised as follows 
	\beq\label{optmeas}
		\gamma =(\text{Id} \times Du)_{\#}f\chi_{U}, 
	\eeq
where $u$ is a convex solution to \eqref{mao} and $U:=\{u>\frac{1}{2}|x|^2\}\cap \Omega.$

  In \cite{CM}, the authors have established the interior $C^{1,\alpha}$ regularity of free boundary 
 $\mathcal{F}:=\partial \{u>\frac{1}{2}|x|^2\}\cap \Omega$,  assuming the domains are strictly convex and densities are bounded from below and above. Then, in \cite{CLWf}, the authors proved the interior higher order regularity of free boundary assuming the domains are smooth uniformly convex and densities are positive smooth. For the global regularity of $\mathcal{F}$ or $u,$
  it was proved in \cite{CM} that the free boundary is $C^1$ regular up to its intersection with the fixed boundary, and that $Du$ is a $C^0$ homeomorphism between $\overline{U}$ and $\overline{\Omega^*}.$ However, nothing more is known concerning the finer regularity of free boundary up to its intersection with the fixed boundary and the global regularity of the solution $u$ in $U.$ There are many basic questions left open, for instance it is even not known whether the active region is a Lipschitz domain.
  
  In this work, we investigate the above mentioned questions systematically. As an initial step, we first prove that the free boundary always intersects the fixed boundary in a nice way.
  Indeed, it is mentioned in \cite[below (7.2)]{CM} that it is not clear that whether the  nontransverse intersection points between the free boundary and the fixed boundary exist or not. In \cite{CM}, the set of  nontransverse intersection points
is specified as
\begin{equation}\label{ntpoint}
\partial_{nt}\Omega:=
\{x\in \partial\Omega\cap\overline{\Omega\cap\partial U} :  \langle Du(x)-x, z \rangle \leq 0\ \ \forall z\in \Omega\}.
\end{equation}
Note that in dimension two, the active domain $U$ looks like a cusp at such a nontransverse intersection point.
It is also pointed out in \cite{CM} that it is not clear whether the free boundary is $C^{1,\alpha'}$ up to such points. 
This issue was later studied by Indrei \cite{I}, in which he proved $\partial_{nt}\Omega$ is of Hausdorff dimension at most $n-2,$ assuming  $\Omega, \Omega^*\in C^{1,1}$ and uniformly convex. 
In this paper we resolve this question as follows.

\begin{theorem}\label{main2}
Let $\Omega, \Omega^* \subset \R^n$ be two bounded, strictly convex domains, with $\overline{\Omega} \cap \overline{\Omega^*}=\emptyset$. Assume that $\lambda^{-1}<f,g<\lambda$  in $\Omega, \Omega^*$ respectively for some positive constant $\lambda.$
Let $u$ be a convex solution to \eqref{mao}.
 Then, $\partial_{nt}\Omega=\emptyset$, in particular, $\partial U$ is globally Lipschitz.
\end{theorem}

A direct consequence of Theorem \ref{main2} is the $C^{1,\alpha}$ regularity of $\mathcal{F}$ up to its intersection with the fixed boundary.
\begin{corollary}
Under the same conditions as in Theorem \ref{main2}, we have that $u\in C^{1,\alpha}(\overline{U})$ and $\overline{\mathcal{F}}$ is $C^{1,\alpha}$ regular for some $\alpha\in (0, 1).$
\end{corollary}
%\begin{remark}
%\emph{
 %In \cite[Corollary 6.7]{CM}, it was shown that $\overline{U}$  is path-connected. Using our Theorem \ref{main2}, we can also improve this result, indeed, we have that the open set $U$ is path-connected.
 %}
%\end{remark}

\begin{remark}
\emph{
If $\Om$ and $\Om^*$ partly overlap, namely if $\Omega\cap\Omega^*\neq\emptyset$,
Figalli \cite{AFi2, AFi}  proved that $\mathcal F$ and $\mathcal F^*$ are locally $C^1$ smooth 
away from the common region $\Omega\cap\Omega^*$.
Later, Indrei \cite{I}  improved the $C^1$ regularity to $C^{1,\alpha'}$, also away from $\Omega\cap\Omega^*$.
}
\end{remark}

To study the higher regularity of $u$ and $\mathcal{F},$ we further assume that the domains $\Omega, \Omega^*$ are $C^2$ and uniformly convex, and the densities $f\in C^0(\overline{\Omega}), g\in C^0(\overline{\Omega^*}).$
For a given point $x_0\in \partial U,$ denote $y_0:=Du(x_0)\in \partial\Omega^*.$
 If $x_0\in \partial \Omega\setminus \overline{\mathcal{F}},$ then since $\partial U, \partial\Omega^*$ are $C^2$ uniformly convex near $x_0, y_0$ respectively, we can apply the argument in \cite{CLW1} to show that $u$ is $C^{1, 1-\epsilon}$ regular near $x_0$ for $\epsilon>0$ as small as we want.
  If $x_0\in \mathcal{F},$ we can apply the argument in \cite{CLWf} to show the corresponding estimate.
 If $x_0\in \overline{\mathcal{F}}\cap \partial\Omega,$ then $U$ is only Lipschitz at $x_0,$ moreover $\partial U$ contains fixed boundary part (convex) and free boundary part (non-convex).  Low regularity and nonconvexity of $\partial U$ near $x_0$ bring serious difficulty to prove the $C^{1, 1-\epsilon}$ regularity of $u$. In this paper, we overcome this difficulty by developing a very delicate blow-up argument at such points, and finally we establish the following theorem.  
  \begin{theorem}\label{thmW2p}
  Besides the assumptions in Theorem \ref{main2}, assume further that $\Omega, \Omega^*$ are $C^2$ and uniformly convex, and that the densities $f\in C^0(\overline{\Omega}), g\in C^0(\overline{\Omega^*}).$ Then $u\in W^{2,p}(\overline{U})$ for any $p>1.$ Moreover, the
  free boundary $\mathcal{F}$ is $C^{1, 1-\epsilon}$ up to its intersection with the fixed boundary for any $\epsilon\in (0, 1).$
  \end{theorem}

\begin{remark}
\emph{In the last section we will provide some example to show that the $W^{2,p}$ regularity of $u$ in
$\overline{U}$ is optimal in some sense. The ideas introduced in this paper can also be used to study the optimal transport problem when the target consists of two separated convex domains, which will be done in a forthcoming paper.
} 
\end{remark}

The rest of paper is organised as follows. 
In Section 2 we recall some results from \cite{C96, CM, CLW1} which will be used in subsequent sections.
In Section 3 we prove the Lipschitz regularity of the active region $U.$
Section 4, we establish the obliqueness estimate for points nearby the intersection of the free boundary and fixed boundary.
In section 5 and 6, we prove the global $C^{1, 1-\epsilon}$ estimate of $u$ and then establish the global $W^{2,p}$ estimate of $u$ in $U.$
In the last section, we construct some example showing that our result is sharp.

%This paper is organised as follows. 
%In \S\ref{S2} we recall some results from \cite{C96, CM, CLW1} which will be used in subsequent sections.
%In \S\ref{S5} we prove the $C^{1,1-\epsilon}$ regularity of the free boundary $\mathcal{F}$ for any given small $\epsilon\in(0,1),$
%assuming the uniform obliqueness condition.
%In \S\ref{S6}, we  raise the $C^{1,1-\epsilon}$ regularity to $C^{2,\alpha}$ by a perturbation method and thus prove Theorem \ref{main1}. \S\ref{Sbu} deals with the blow-up analysis at the free boundary where the obliqueness fails, which leads to a contradiction in \S\ref{Sob}  and thus confirming the obliqueness property.  {\color{blue} In \S\ref{S7} , we use the new method developed in section 5 to prove Theorem \ref{main2}. }

\vskip10pt
 
\section{Preliminaries}\label{S2}

\subsection{Potential functions}\label{s21}
Throughout the paper, we always assume that the densities $f, g$ satisfy
$\lambda^{-1}<f, g<\lambda$ in $\Omega, \Omega^*$, respectively, for a positive constant $\lambda$,  
and the domains  $\Omega, \Omega^*\subset \mathbb{R}^n$ are bounded, strictly convex and their closures are disjoint. The source is assumed to contain more mass than the target, namely,
$\int_{\Omega}f\geq \int_{\Omega^*}g.$ The active region is denoted by $U:=\left\{u> \frac{1}{2}|x|^2 \right\}\cap \Omega.$ 

Let $u:\mathbb{R}^n\rightarrow \mathbb{R}$ be a convex solution to \eqref{mao}. Replacing 
$u$ by $\max\{u, \frac{|x|^2}{2}\},$ we may assume that $u=\frac{|x|^2}{2}$ in $\Omega\setminus \overline{U}.$
Let $$v(y):= \sup_{x\in \Omega}  \left\{  y\cdot x-u(x) \right\} \  \ \text{for}\ y\in \mathbb{R}^n.$$
Then, 
 \begin{equation} \label{push}
 (Dv)_{\#}(g+f\chi_{_{\Omega\backslash \overline{U}}})) = f.
 \end{equation}
 We may also extend $u$ from $U$ to $\mathbb{R}^n$ in the following way
 \begin{equation}\label{pushu}
 \bar{u}(x):= \sup\{L(x): L\ \text{is affine},\ L\leq u\ \text{in}\ U, \text{and}\ DL\in \Omega^*\}.
 \end{equation}

By \eqref{push}, \eqref{pushu} and since  $\Omega, \Omega^*$ are bounded and convex, 
 $u, v$ are globally Lipschitz in $\R^n$ and satisfy 
\beq\label{Asolv} 
	C^{-1}(\chi_{_{\Om\setminus U}}+\chi_{_{\Omega^*}}) 
	      \leq \det\, D^2 v \leq C (\chi_{_{\Om\setminus U}}+\chi_{_{\Omega^*}})
  \eeq
  and 
  \beq\label{Asolv1} 
	C^{-1}\chi_{_{U}}
	      \leq \det\, D^2 \bar{u} \leq C\chi_{_{U}}
	        \eeq
in the sense of Alexandrov \cite{C92}, where $C$ is a positive constant depending only on $\lambda$.

For a convex function $w : \R^n\to(-\infty, \infty]$,
the associated \emph{Monge-Amp\`ere measure} $\mu_w$ is defined by
\beq\label{MAmeas}
\mu_w(E) := \left| \p w(E) \right| 
\eeq 
for any measurable set  $ E\subset\R^n$,
where $\p w$ is the sub-gradient of $w$ and $|\cdot|$ denotes the $n$-dimensional Hausdorff measure.
%If $w$ is $C^2$ smooth, then
%$$ \mu_w(E) = \int_E \det\, D^2w(x) \,dx. $$ 
We say that $w$ satisfies $C_1 \chi_{_W}\leq \det\, D^2 w \leq C_2\chi_{_W}$ in the sense of Alexandrov,
if
$$ C_1|E\cap W|\leq \mu_w(E) \leq C_2|E\cap W| \ \ \ \forall \ E\subset\R^n. $$
Hence \eqref{Asolv} implies that the Monge-Amp\`ere measure 
$\mu_v$ is actually supported and bounded on $(\bom\setminus U)\cup \overline {\Omega^*}.$

\vskip5pt

\subsection{$C^{1,\alpha'}$ regularity of $\mathcal{F}$}
We recall the \emph{interior ball condition} proved in \cite{CM}, 
which will be useful in our subsequent analysis. 

\begin{lemma}[\cite{CM}]\label{intball}
Let $x\in U$ and $y=Du(x)$, then 
	\[ \Om\cap B_{|x-y|}(y) \subset U. \]
%Likewise, let $y\in V$ and $x=Dv(y)$, then
%	\[ \Lm\cap B_{|x-y|}(x) \subset V. \]
\end{lemma}

It is shown in \cite{CM} that $u$ is $C^1$ smooth up to the free boundary $\mathcal F$, 
and the unit inner normal vector of $\mathcal F$ is given by
 \begin{equation}\label{normalformular}
 \nu(x)=\frac{Du(x)-x}{|Du(x)-x|}\quad \forall\ x\in \mathcal F.
 \end{equation}
Hence, the regularity of $u$ up to the free boundary $\mathcal F$ implies the regularity of the free boundary $\mathcal F$ itself.
The following regularity results have been obtained in \cite{CM}.

\begin{theorem}[\cite{CM}] \label{CMCL}
Assume that $\Omega, \Omega^*$ are disjoint and strictly convex,
 the densities $f, g$ satisfy
$\lambda^{-1}<f, g<\lambda$ for a positive constant $\lambda$. 
Then
\begin{itemize}

\item [$i)$]  $\bar u, v\in C^1(\mathbb{R}^n)$,  $Dv$  is 1-1 from $\overline{\Omega^*}$ to $\U$, and $Du$  is 1-1 from $\overline{U}$ to $\overline{\Omega^*}$. 
 
\item [$ii)$]  $u\in C^{1,\alpha'}$ up to the free boundary $\mathcal{F}$, and $\mathcal{F}$ is $C^{1,\alpha'}$ for some $\alpha\in(0,1)$.

%\item [$iii)$]  $\forall\,x_0\in \mathcal{F}$, $\exists$ a neighborhood $\mathcal{N}$ of $x_0$ such that $v$ is strictly convex in $Du(\mathcal{N}\cap\overline{U})$.

%\item [$iv)$]   Let $y_0=Du(x_0)$. Then $y_0\in \partial V\setminus \overline{\partial V\cap \Omega^*}\subset \partial\Omega^*.$
%Moreover, there exists a constant $r$ depending on $\text{dist}(x_0, \partial \Omega),$ such that $B_r(y_0)\cap \Omega^*\subset V.$

\end{itemize}
\end{theorem}

\vskip5pt

\subsection{Sub-level sets} \label{s23}
To study higher order regularity of the potentials $u, v$,
we introduce the (centred)  sub-level sets as in  \cite{C92b,C96}.
%Note that from $iii)$ and $ iv)$ of Theorem \ref{CMCL}, the function $v$ is locally strictly convex 
%near $Du(\mathcal{F}) \subset \partial V\setminus \overline{\partial V\cap \Omega^*}$, 
%which (as a portion of $\partial\Omega^*$) is convex as well. 

\begin{definition}\label{defS}
Let $y_0\in \overline{\Omega^*}$ and $h>0$ be a small constant. We denote by
\beq\label{sect}
	S^c_{h}[v](y_0) := \left\{y\in\R^n \,:\, v(y)< v(y_0) + (y-y_0)\cdot \bar{p} + h\right\}
\eeq
the centred sub-level set of $v$ with height $h$, where $\bar{p}\in \R^n$ is chosen 
such that the centre of mass of $S^c_{h}[v](y_0)$ is $y_0$.
We denote by
\beq\label{sub}
S_h[v](y_0) : =\left\{y\in \Omega^* \,:\, v(y) < \ell_{y_0}(y) + h\right\}
\eeq
the sub-level set of $v$ with height $h$, where $\ell_{y_0}$ is a support function of $v$ at $y_0$. 
\end{definition}

Note that in the above definition, 
$S_h[v](y_0)$ is a subset of $\Omega^*$ but $S^c_{h}[v](y_0)$ may not be contained in $\Om^*$.
In the following we will write $S_h[v](y_0)$ and $S^c_{h}[v](y_0)$ as 
$S_h[v]$  and $S^c_{h}[v]$ when no confusion arises.

\begin{remark} \label{uniest11}
\emph{
Suppose $v(0)=0, v\geq 0.$
Let $L$ be the affine function such that $S^c_h[v](0)=\{v<L\}$.  
Since $(L-v)(0)=h$,  $L=v$ on $\partial S^c_h[v](0),$ $L\geq v\geq  0$ in $S^c_h[v](0),$ 
and $S^c_h[v](0)$ is balanced around $0$,  we have that 
 \begin{equation}\label{secrela}
v\leq L\leq Ch\ \ \ \text{in}\  S^c_h[v](0)
\end{equation}
 for a constant $C$ depending only on $n.$ 
The same property also holds if $v$ is replaced by $u.$ }
\end{remark}

%Let $L$ be the affine function such that $S^c_h[v](y_0)=\{v<L\}$.   Since $(L-v)(y_0)=h$,  $L=v$ on $\partial S^c_h[u](y_0)$, $L\geq v$ in $S^c_h[v](y_0)$, and $S^c_h[v](y_0)$ is balanced around $y_0$,  we have 
% \begin{equation}\label{secrela}
%0\leq L-v \leq Ch\quad\text{ in }S^c_h[v](y_0),
%\end{equation}
%for some constant $C$ depending only on $n$.
%In particular, if $v(y_0)=0, v\geq 0,$ then 
%\begin{equation}
%v \leq Ch\quad\text{ in }S^c_h[v](y_0)

For any $x_0\in \mathcal{F}$,  we have  $y_0:=Du(x_0)\in \partial\Om^*$.
When $h>0$ is sufficiently small, by \cite[Lemma 7.6 and Lemma 7.11]{CM} we have 
\begin{equation}\label{localise}
	 S_h^c[v] (y_0) \cap \overline{\Omega}=\emptyset. 
\end{equation}
By \cite[Theorem 7.13]{CM} we have furthermore the strict convexity 
 \begin{equation}\label{strictconvex}
 v(y)\geq v(y_0)+Dv(y_0)\cdot (y-y_0)+ C|y-y_0|^{1+\beta}\quad \forall\, y\in \overline{\Omega^*}\text{ near } y_0
 \end{equation}
for some constant  $\beta>1$, which in turn implies $u\in C^{1,\alpha'}$ as in part $ii)$ of Theorem \ref{CMCL}. 
% {\color{purple} (later, we need to make $\alpha, \beta, \alpha'$ consistent throughout the paper. )}

\vskip5pt

%The above uniform density was proved in \cite[Theorem 3.1]{C96} under the condition that the source domain is polynomial convex and the target domain is convex.
%Here we consider the potential $v$ in the domain $V$, and $V$ is uniformly convex near $y_0$, 
%which is stronger than the polynomial convexity. 
%But the target $U$ may not be convex near $x_0=Dv(y_0) \in\mathcal{F}$.
%Thanks to the $C^{1,\alpha'}$ regularity of $\mathcal F$ in $ii)$ of Theorem \ref{CMCL}, 
%we are able to work out a proof based on that in \cite{C96}.

%%% From \eqref{ima1} we know that $Dv(S_h^c[v])$ is similar to the ellipsoid $\tilde E$ centred  
In this paper,  
the notation $a\lesssim b$ (resp. $a\gtrsim b$) means that there exists a constant $C>0$ independent of $h$ and the potential functions $u$ and $v$,
such that $a\leq Cb$ (resp. $a\geq Cb$), and the notation $a\approx b$ means that $C^{-1}a\leq b\leq Ca$, where $a, b$ are both positive constants.
Given a convex domain $D\subset\R^n$, 
we say that $D$ has a good shape if the eccentricity of its minimum ellipsoid is uniformly bounded.

%\begin{notation}\label{uvxy}
%For clarity, we will use $x$ for variables of $u$ and $y$ for variables of $v$. 
%Accordingly, we will use $x$ to denote points in $S_h[u]$ and $y$ for points in $S_h[v]$.
%But $x, y$ are in the same coordinates of  $\R^n$.
%We also use $p, q,$ etc to denote a point in $\R^n,$ and use $e_i$ to denote the unit vector on the positive $x_i$-axis for $x=1,\cdots, n.$
%\end{notation}

\section{A localisation lemma for $v$} \label{S7} 
 \begin{proof}[Proof of Theorem \ref{main2}]
 Suppose $\partial_{nt}\Omega$ is not empty, and let $x_0$ be a point in  $\partial_{nt}\Omega.$
 Denote $y_0=Du(x_0).$ Without loss of generality we may assume $x_0=0.$ Denote
 $\nu=\frac{Du(0)}{|Du(0)|}.$ Up to a rotation of coordinates we may assume $\nu=e_n$ is the $n$-th coordinate direction. By the definition of $\partial_{nt}\Omega$ we have that 
 \begin{equation}\label{np4}
 \Omega\subset \{x_n\leq 0\}.
 \end{equation}
 
 Let $G_{\epsilon}:=B_{\epsilon}(y_0)\cap\Omega^*,$ where
$\epsilon$ is a small positive number to be determined later. 
% the interior ball property we have that $G_{\epsilon}\subset V.$
For any $p\in G_{\epsilon},$ let $q=Dv(p).$
By  the interior ball property, we have 
\begin{equation}\label{np5}
|p-q|\leq |p-0|,
\end{equation}
 since otherwise $0\in B_{|p-q|}(p)\cap \Omega\subset U$ contradicts to
the assumption that $0\in \partial\Omega\cap\overline{\Omega\cap\partial U}.$  
%Note also that by strict convexity 
%of $\Omega^*,$ we have that $p$ is close to $y_0$ when $\epsilon$ is small. Hence $p_n\approx r.$
By \eqref{np4} and \eqref{np5} we have that 
\begin{eqnarray*}
|q_n|&\leq& |p-0|-p_n\\
&=& p_n\left(1+\frac{\sum_{i=1}^{n-1}p_i^2}{p^2_n}\right)^{1/2}-p_n\\
&=&\frac{\sum_{i=1}^{n-1}p_i^2}{2p_n}+o(\frac{\sum_{i=1}^{n-1}p_i^2}{2p_n}).
\end{eqnarray*}
Then, it follows from the definition of $G_{\epsilon}$ and the above estimate that 
\begin{equation}\label{np6}
|q_n| \leq C\epsilon^2.
\end{equation}

By \eqref{np5} again, we have that
$$\sum_{i=1}^n|p_i-q_i|^2 \leq \sum_{i=1}^np_i^2.$$
This implies 
$$\sum_{i=1}^nq_i^2\le\sum_{i=1}^n2p_iq_i\leq \sum_{i=1}^{n-1}2p_iq_i,$$
where we have used the fact that $p_n>0, q_n<0.$
Hence, by Cauchy-Schwarz inequality we have 
$$\sum_{i=1}^{n-1}q_i^2\leq 2(\sum_{i=1}^{n-1}p_i^2)^{\frac{1}{2}}(\sum_{i=1}^{n-1}q_i^2)^{\frac{1}{2}},$$ which implies $\sum_{i=1}^{n-1}q_i^2\leq 4\sum_{i=1}^{n-1}p_i^2\leq 4\epsilon^2.$ 
Hence 
\begin{equation}\label{np7}
|q_i|\leq 2\epsilon\ \ \text{ for}\ i=1, \cdots, n-1.
\end{equation}

By \eqref{np6} and \eqref{np7} we have that 
\begin{equation}\label{np8}
|Dv(G_{\epsilon})| \leq C\epsilon^{n+1}.
\end{equation}
On the other hand, since $\Omega^*$ is convex,  by the definition of $G_{\epsilon},$
 we have that $|G_{\epsilon}| \gtrsim  \epsilon^n.$
Hence $|G_{\epsilon}| \gg |Dv(G_{\epsilon})|$ as $\epsilon$ is sufficiently small, which contradicts to the fact that $Dv$ is the optimal transport map between $U$ and $\Omega^*$ with densities bounded between $1/\lambda$ and $\lambda$ for some positive constant $\lambda.$
Therefore, $\partial_{nt}\Omega$ must be an empty set. Finally, by a standard covering argument we have that $U$ is globally Lipschitz.
\end{proof}

Combining  Theorem \ref{main2} and \cite{CM} we have the following useful localisation lemma for $v.$
\begin{lemma}\label{locallemma1}
There exists $h_0>0$ small such that
for any $y \in \overline{\Omega^*}$ we have $S_h^c[v](y)\cap \overline\Omega=\emptyset,$ 
provided $h\leq h_0.$
\end{lemma}
Note that Lemma \ref{locallemma1} is a strengthened version of \eqref{localise}.
Then by Theorem \ref{main2} and \cite{C96,CLWf} we have the following important properties of $v.$
\begin{lemma}[Uniform density] \label{ud1} 
Let $\Omega, \Omega^*, f, g$ be as in Theorem \ref{main2}. Assume that $\Omega^*$ is $C^2$ and uniformly convex. 
For any $y_0\in \partial \Omega^*,$
we have 
\beq\label{UniDen}
{\small{\text{$\frac{|S_h^c[v](y_0) \cap \Omega^*|}{|S_h^c[v](y_0) |}$}}}  \geq \delta ,
\eeq
provided $0<h\leq h_0,$
where $\delta$ is a positive constant depending on $n, \lambda,\Omega^*,$
but independent of $h$.  
 \end{lemma} 

\begin{corollary}\label{co21}
Under the conditions in Lemma \ref{ud1}, we have
\begin{itemize}
\item[$(i)$] \emph{Volume estimate:}
	\begin{equation}\label{nm1}
 		|S_h[v](y_0)|\approx |S_h^c[v](y_0)\cap \Omega^*|\approx |S_h^c[v](y_0)|\approx h^{\frac{n}{2}}.
	\end{equation} 
Moreover, for any given affine transform $\mathcal{A}$, 
if one of $\mathcal{A}(S_h^c[v](y_0))$ and $\mathcal{A}(S_h[v](y_0))$ has a good shape, so is the other one.

\item[$(ii)$] \emph{Tangential $C^{1,1-\epsilon}$ regularity for $v$:} Assume in addition that $f\in C(\overline{\Omega}),\ g\in C(\overline{\Omega^*})$. Let $\mathcal{H}$ be the tangent hyperplane of $\partial\Omega^*$ at $y_0$.   
  Then $\forall\,\epsilon>0$,  $\exists \,C_\epsilon$ such that 
 \begin{equation}\label{tanc2}
 B_{C_\epsilon h^{\frac{1}{2}+\epsilon}}(y_0) \cap \mathcal{H} \subset S_h^c[v](y_0) \ \ \ \text{for $h>0$ small.} 
 \end{equation}
%{\color{red} here $0$ should also be replaced by $y_0$?} 
%\item[$(iii)$] The estimate \eqref{tanc2} also holds in the sub-level $S_h[v]$, namely 
%\beq
 %B_{C_\epsilon h^{\frac{1}{2}+\epsilon}}(0)\cap \partial V \subset S_h[v]\quad \ \text{for $h>0$ small.} 
 %\eeq
\end{itemize}
\end{corollary}

\begin{proof}  
Assume $y_0=0$
and write $S_h^c[v](0), S_h[v](0)$ as $S_h^c[v], S_h[v]$ for brevity. 
By the strict convexity estimate of $v$  in $\overline{\Omega^*}$ (see \eqref{strictconvex}) and the fact that $S^c_h[v]$ is balanced around $0$,
we have an equivalence relation between $S_h[v]$ and $S^c_{h}[v]$: 
\beq\label{equi0} 
S^c_{b^{-1}h}[v] \cap \Om^* \subset S_h[v] \subset S^c_{bh}[v] \cap \Om^*\ \ \ \forall\, h>0 \text{ small},
\eeq
where $b\geq1$ is a constant independent of $h$. 
For a proof of \eqref{equi0}, we refer the reader to  \cite[Lemma 2.2]{CLW1}.
%{\Small\color{blue}{(the strict convexity of $v$ seems not a main reason for \eqref{equi0} )}}

From Lemma \ref{ud1} and \eqref{equi0}, the volume estimate \eqref{nm1} 
can be deduced similarly as in \cite[Corollary 3.1]{C96}. 
Note that by \eqref{localise} we have that $\det\, D^2v=\tilde{f}(y)\chi_{S^c_h[v]\cap\Omega^*}$
in $S_h^c[v],$ where $\tilde{f}(y)=\frac{g(y)}{f(Dv(y))}\in C(S^c_h[v]\cap \overline{\Omega^*}).$ 
Then, the proof of tangential $C^{1,1-\epsilon}$ estimate  is the same as in \cite[Lemma 4.1]{C96}.
\end{proof}

\section{Obliqueness at intersection points}
In this section we prove the obliqueness estimate at the intersection points of the free boundary $\overline{\mathcal{F}}$ and the fixed boundary $\partial\Om$. Given $y_0\in \partial\Omega^*,$ let $x_0:=Dv(y_0).$ If $x_0\in \partial \Omega\setminus \overline{\mathcal{F}}$ (resp. $x_0\in\mathcal{F}$) the obliqueness estimate has been established in \cite{CLW1} (resp. \cite{CLWf}). The situation becomes more complicated when $x_0\in \partial\Omega\cap \overline{\mathcal{F}}.$ 
Denote by $\nu_{\Omega}(x_0), \nu_{\Omega^*}(y_0) $ the inner unit normal of $\partial \Omega, \partial \Omega^*$ at $x_0, y_0$ respectively,  and denote by
$\nu_{\overline{\mathcal{F}}}(x_0):=\frac{Du(x_0)-x_0}{|Du(x_0)-x_0|}$ the unit normal of $\overline{\mathcal{F}}$ at $x_0$ in the direction of transportation. We have the following key estimate of this work.
\begin{proposition}\label{obli}
Suppose the hypotheses of Theorem \ref{thmW2p}.  Let $y_0\in \partial\Omega^*,$ and $x_0:=Dv(x_0)\in \partial\Omega\cap \overline{\mathcal{F}}.$ 
Then $\nu_{\Omega^*}(y_0) \cdot \nu_{\overline{\mathcal{F}}}(x_0)>0$ and $\nu_{\Omega^*}(y_0) \cdot \nu_{\Omega}(x_0)>0.$
\end{proposition}

 We need to rule out two cases: {\bf (i)} $\nu_{\Omega^*}(y_0) \cdot \nu_{\overline{\mathcal{F}}}(x_0)=0$; and {\bf (ii)} $\nu_{\Omega^*}(y_0) \cdot \nu_{\Omega}(x_0)=0.$

\subsection{Case (i)}
By a translation of coordinates we may assume $x_0=0.$
Denote by $e_i$, $i=1, \cdots, n$ the standard coordinate directions of $\mathbb{R}^n.$
 Up to a rotation of coordinates we may assume $\nu_{\overline{\mathcal{F}}}(0) =e_n$ and $\nu_{\Omega^*}(y_0)=e_1.$ Denote by $H_{\mathcal{F}}(0)$ the tangent hyperplane of $\overline{\mathcal{F}}$ at $0.$ Denote by $H_{\Omega}(0)$ the tangent hyperplane of $\Omega$ at $0.$ 
Let $H':=H_{\mathcal{F}}(0)\cap H_{\Omega}(0).$ 
%we have that $H'$ is an $n-2$-dimensional subspace of $\mathbb{R}^n.$ Let $e'_i$, $i=2,\cdots, n-1$ be $n-2$ orthonormal unit vectors such that $H'=\text{span}\{e'_2,\cdots, e'_{n-1}\}.$
By the interior ball property, we have that 
\begin{equation}\label{itf1}
B_{|y_0|}(y_0)\cap \Omega\subset U.
\end{equation}
 Note that $H_{\mathcal{F}}(0)$ is also tangent to $\partial B_{|y_0|}(y_0).$ By Theorem \ref{main2}, we have 
 \begin{equation}\label{itf2}
 \nu_{\Omega}(0) \cdot \nu_{\overline{\mathcal{F}}}(0)>-1.
 \end{equation}
 We will prove the following lemma.
\begin{lemma}\label{localgood}
There exists a $n-2$ dimensional $C^2$ submanifold   of $\mathbb{R}^n$ (denoted by $\mathcal{M}$)  passing through $0,$ such that its tangent space at $0$ is contained in $H'$ and 
$\mathcal{M}\cap B_{r_0}(0)\subset \overline{U}$ for some small $r_0>0.$
\end{lemma} 
\begin{proof}
If $ \nu_{\Omega}(0) \cdot \nu_{\overline{\mathcal{F}}}(0)=1,$  since 
$\partial \Omega$ is $C^2,$ by \eqref{itf2} we have that $B_{r_0}(0)\cap \{x_n>C|x'|^2\}\subset U$ for some small positive constant $r_0$ and some large constant $C$ depending only the $C^2$ norm of $\partial \Omega$ and the distance between $\Omega$ and $\Omega^*.$ 
Hence we can take $\mathcal{M}= \{x_n=C|x'|^2\}.$

If $-1<\nu_{\Omega}(0) \cdot \nu_{\overline{\mathcal{F}}}(0)<1,$ 
since both $\partial \Omega$ and $\partial B_{|y_0|}(y_0)$ are $C^2,$ by implicit function theorem  
we have that $\mathcal{M}:=\partial\Omega\cap\partial B_{|y_0|}(y_0)$ is an $n-2$ dimensional $C^2$ submanifold of $\mathbb{R}^n.$ By \eqref{itf1} we have that  $\mathcal{M}\cap B_{r_0}(0)\subset \overline{U}$ for a small $r_0>0$.  

Therefore, we can always find the desired $\mathcal{M}$ as in the statement of the lemma.
\end{proof}

By subtracting a constant we can also assume that $v\geq 0$ and $v(y_0)=0$.
Since $\overline{\Omega}\cap\overline{\Omega^*}=\emptyset,$ we have $y_0=re_n$ for some $r>0$.  
Let $p=(p_1, 0,\cdots, 0, p_n)$ be a point on $\partial\{v<h\}\cap \partial \Omega^*$ with $p_n<r.$ 
Denote $s=r-p_n$.    
Since $\partial \Omega^*$ is $C^2$ smooth and uniformly convex, 
we have  $p_1=as^2+o(s^2)$ for a positive constant $a.$
Similar to \cite[Lemma 5.9]{CLWf}, we have the following estimate for $s.$
\begin{lemma}\label{estfors}
For any $\epsilon>0$ small, there exist constants $C, C_\epsilon$ such that
\beq\label{hsh}
C h^{\frac{1}{3}}\leq s\leq C_ \epsilon h^{\frac{1}{3}-\epsilon}
\eeq
when $h>0$ is small, where $C>0$ is a constant independent of $\eps$.
\end{lemma} 

\begin{proof}
The proof is a small modification of that of \cite[Lemma 5.9]{CLWf}. We explain the necessary change here. For the second inequality, the proof is exactly same as that of \cite[Lemma 5.9]{CLWf}.
It suffices to show the first inequality of \eqref{hsh}.
Since $v\in C^1(\mathbb{R}^n)$ and $Dv=\text{Id}$ in $\Omega\setminus U$.
Hence, as $0\in\mathcal{F}\subset\partial U$,
$$Dv(0)=0=Dv(y_0). $$  
By the convexity of $v$, we infer that
$$ Dv(te_n)=0\quad \forall\,t\in[0,r].$$ 
Since  $v(y_0)=0$ and $v\geq 0$ on $\mathbb{R}^n$.
Then $v(te_n)=0$ for all $t\in[0,r]$ as well.  
In particular, we have $v(z)=0$,
where $z=p_ne_n$ is the projection of $p$ on the $x_n$ axis.
Denote $q=(q_1,\cdots, q_n) =Dv(p)\in \partial U$.
By the convexity of $v$, we have 
\begin{equation}\label{h4}
q_1=Dv(p)\cdot e_1\geq \frac{v(p)-v(z)}{|p-z|}=\frac{h}{p_1}\geq C\frac{h}{s^2} .
\end{equation}
By the interior ball property  (Lemma \ref{intball}),  we have $B_{|p-q|}(p)\cap \Omega\subset U$. 
Hence
\begin{equation}\label{h5}
|p-q|^2\leq |p-0|^2.
\end{equation}
By monotonicity of $Dv,$ we have that $(p-y_0)\cdot (q-0)\geq 0,$ which implies
\begin{equation}\label{hh1}
q_n\leq \frac{1}{s}p_1q_1.
\end{equation}
Note that this part is different from that in the proof of \cite[Lemma 5.9]{CLWf}.

Denote $q'=(q_1,q_2,\cdots,q_{n-1})$.  
Recall that $p_n=r-s$. By \eqref{h5} and \eqref{hh1}, we have
$$|q'|^2+p_1^2-2p_1q_1+\big(r-s-\frac{1}{s}p_1q_1)\big)^2\leq p_1^2+(r-s)^2, $$
from which one infers that $\frac{s}{r}|q'|^2\le 2 p_1q_1$. 
Noting that $q_1\le |q'|$, we thus obtain  
\begin{equation*}
	\frac{s}{r}q_1 \leq 2p_1.
\end{equation*}
Recall that $p_1 \leq Cs^2+o(s^2)$. By \eqref{h4}, we then deduce
$$ \frac{h}{s r} \leq Cp_1\le Cs^2,$$
from which it follows that $s\geq Ch^{\frac{1}{3}}$. 
So the first inequality of \eqref{hsh} is proved.
\end{proof}

Up to a translation of coordinates we may also assume $y_0=0.$ 
By subtracting a constant we may assume $v(0)=0, v\geq 0.$
Since
$\nu_{\Omega^*}(0) \cdot \nu_{\overline{\mathcal{F}}}(0)=0,$ we can apply the argument in \cite[Section 5.2]{CLWf} to show that $S_h^c[v](0)$ can be normalised by an affine transformation
$T_h,$ namely,  $T_h(S_h^c[v]) \sim B_1(0).$ Moreover, $T_h=T_2\circ T_1,$ where $T_2$ is an affine transform satisfying  
\begin{equation}\label{ag2}
 \|T_2\|+\|T_2^{-1}\|\leq C_\epsilon h^{-\epsilon}\quad\mbox{ for any } \epsilon>0,
 \end{equation}
 and $T_1:y \mapsto \bar y$ is the transform given by
\beq
\label{T1}
\left\{ 
{\begin{array}{ll}
  \bar y_1=h^{-\frac23}y_1,& \\
  \bar y_i= h^{-\frac12}y_i, & \ \ \ i=2,\cdots, n-1,\\
  \bar y_n=h^{-\frac13}y_n. &
  \end{array}}\right. 
\eeq

If $ \nu_{\Omega}(0) \cdot \nu_{\overline{\mathcal{F}}}(0)=1,$ by the uniform convexity and $C^2$ regularity of $\partial \Omega$, and the interior ball property we have that 
$\partial U=\{x_n=\rho(x')\}$ near 0 for some function $\rho$ satisfying $C^{-1}|x'|^2\leq \rho(x')\leq C|x'|^2,$ then
 we can use the blow up argument in 
\cite[Section 5]{CLWf} to make a contradiction. 

In the following we only need to consider the situation
$ -1<\nu_{\Omega}(0) \cdot \nu_{\overline{\mathcal{F}}}(0)<1,$ in which $H':=H_{\mathcal{F}}(0)\cap H_{\Omega}(0)$ is an $n-2$-dimensional subspace of $\mathbb{R}^n.$ 
%Let $e'_i$, $i=2,\cdots, n-1$ be $n-2$ orthonormal unit vectors such that $H'=\text{span}\{e'_2,\cdots, e'_{n-1}\}.$
We need to consider two subcases: 1) $e_1\in H';$ and 2) $e_1\notin H'.$

\subsubsection{Subcase 1: $e_1\in H'$}  In this subcase up a change of of $e_2, \cdots, e_{n-1}$ coordinates we may assume $H'=\text{span}\{e_1, e_3,\cdots, e_{n-1}\}.$
For any unit vector $e\in \text{span}\{e_2, \cdots, e_{n-1}\},$ let 
$p_h:=h^{\frac{1}{2}-3\epsilon}e+\rho^*(h^{\frac{1}{2}-3\epsilon}e)e_1\in \partial\Omega^*,$ where 
$\Omega^*=\{y_1>\rho^*(y_2, \cdots, y_n)\}$ near $0$ for some $C^2,$ uniformly convex function $\rho^*$ with $\rho^*(0)=0, D\rho^*(0)=0.$ 
 A direct computation shows that 
 \begin{equation}\label{blimit1}
 |T_hp_h-T_h(h^{\frac{1}{2}-3\epsilon}e)|\rightarrow 0,\ \text{as}\ h\rightarrow 0,
 \end{equation}
 and that
 \begin{equation}\label{blimit2}
 |T_h(h^{\frac{1}{2}-3\epsilon}e)|\rightarrow \infty,\ \text{as}\ h\rightarrow 0.
 \end{equation}
 By the Blaschke selection theorem and the standard technique of taking diagonal sequences,
  we have that the convex sets $T_h(\Omega^*)$ locally uniformly converges to a limit convex set
  $\Omega^*_0$ in Hausdorff distance. Up to a subsequence we may also assume that  $T_h(\text{span}\{e_2,\cdots, e_{n-1}\})$ converges to $H^*,$ which is an $n-2$ dimensional subspace of $\mathbb{R}^n.$ 
  By \eqref{blimit1} and \eqref{blimit2}, we have that $H^*\subset \overline{\Omega^*_0}.$
  Hence by convexity we have that the convex set $\Omega^*_0$ splits, namely,
  \begin{equation}\label{split01}
  \Omega^*_0=H^*\times \omega^*
  \end{equation}
  for some two dimensional convex set $\omega^*.$
  Let
	\begin{equation}\label{novh1}
		v_h(y):=\frac{1}{h}v(T_h^{-1}y).
	\end{equation}
Then, $v_h$ is locally uniformly bounded in $\R^n$ as $h\to0$. 
Hence by passing to a subsequence, $v_h\rightarrow v_0$, locally uniformly,
and $v_0$ satisfies 
\begin{equation}\label{maeq111}
\det\, D^2 v_0=c_0\chi_{_{\Omega^*_0}}\quad\text{in } \R^n
\end{equation}
for a constant $c_0>0$.
Since $v_0$ is a convex function defined on entire $\mathbb{R}^n,$ we have that
$U_0,$ the interior of $\partial v_0 (\mathbb{R}^n)$ is a convex set. Since $\Omega^*_0$ is convex, we have that $c_0\chi_{_{\Omega^*_0}}$ is doubling for any convex set centred at a point in 
$\overline{\Omega^*_0}.$ Hence, by Caffarelli's boundary regularity theory \cite{C92b}, we have that $v_0$ is $C^1$ and strictly convex on $\overline{\Omega^*_0}.$

Denote $T_h^*:=\frac{1}{h}(T^t_h)^{-1}$. 
A straightforward computation shows that the unit inner normal of $\Omega_h:=T_h^*(\Omega)$ at $0$ is given by
\begin{equation}\label{in11}
\nu_{_{\Omega_h}}(0):=\frac{T_h\nu_{\Omega}(0)}{|T_h\nu_{\Omega}(0)|}.
\end{equation}
Since
$ -1<\nu_{\Omega}(0) \cdot \nu_{\overline{\mathcal{F}}}(0)<1,$ we have that 
$\nu_{\Omega}(0)=(0, c_2, 0, \cdots, 0, c_n)$ for some constants $c_2, c_n\neq 0$ satisfying 
$c_2^2+c_n^2=1.$
By the formula of $T_1,$ we have that $T_1\nu_{\Omega}(0)=c_2h^{-\frac{1}{2}}e_2+c_nh^{-\frac{1}{3}}e_n.$
By \eqref{ag2} we have that
\begin{equation}\label{Thnu1}
\frac{T_h\nu_{\Omega}(0)}{|T_h\nu_{\Omega}(0)|}=\frac{T_2(T_1\nu_{\Omega}(0))}{|T_2(T_1\nu_{\Omega}(0))|}=\frac{T_2e_2}{|T_2e_2|}+o(1).
\end{equation}
Note that $\nu_{_{\Omega_h}}(0)$ is the unit inner normal of $\Omega_h$ at $0.$
Up to a subsequence we may assume that $\frac{T_2e_2}{|T_2e_2|}$ converges to some unit vector $e_0\in H^*$ as $h\rightarrow 0.$ 
By \eqref{Thnu1} we have $\nu_{_{\Omega_h}}(0)$ converges to $e_0$ as $h\rightarrow 0.$
Since $Dv_h(\mathbb{R}^n)\subset \overline{\Omega_h},$ we have that $Dv_h(x)\cdot \nu_{_{\Omega_h}}(0)\geq 0$ for any $x\in \mathbb{R}^n.$ Passing to the limit we have that 
$Dv_0(x)\cdot e_0\geq 0$ for any $x\in \mathbb{R}^n.$ Hence $v_0$ is monotone increasing along $e_0$ direction. Since $v_0(0)=0, v_0\geq 0,$ we have that $v_0(-te_0)=0$ for any $t>0,$ contradicts to the strict convexity of $v_0$ on $\overline{\Omega^*_0}.$

\subsubsection{Subcase 2: $e_1\notin H'$} 
For any unit vector $e\in H',$ let $p_t=te$ be a point on $H'$ with $t\leq h^{\frac{1}{2}-3\epsilon}.$
By Lemma \ref{localgood} we have that for any $t>0$ small, there exists a point 
$q_t\in \mathcal{M}\cap B_{r_0}(0)\subset \overline{U},$ such that $|q_t-p_t|\leq Ct^2\leq Ch^{1-6\epsilon}.$ 
Write 
\begin{equation}\label{limm1}
e=c_1e_1+c_2e_2+\cdots+c_{n-1}e_{n-1}.
\end{equation}
 Since $e_1\notin H',$ we have that at least
one of $c_2, \cdots, c_{n-1}$ is not zero. Without loss of generality we may assume $c_2\neq 0.$
Hence $|(T_1^t)^{-1}p_t|\geq Ch^{-3\epsilon}.$ By \eqref{ag2} we have that 
\begin{equation}\label{lim003}
|T^*_hp_t|\geq Ch^{-2\epsilon}\rightarrow \infty\ \text{as}\ h\rightarrow 0,
\end{equation}
where $T^*_h=\frac{1}{h}(T_h^t)^{-1}$.
By \eqref{T1}, \eqref{ag2} and \eqref{limm1} we have that
\begin{equation}\label{lim004}
|T^*_hq_t-T^*_hp_t|\leq Ch^{\frac{1}{3}-6\epsilon} \rightarrow 0\ \text{as}\ h\rightarrow 0,
\end{equation}
provided $\epsilon$ is chosen small.
Up to a subsequence we may assume $T^*_h H'$ converges $H,$ an $n-2$ dimensional subspace of $\mathbb{R}^n,$ locally uniformly as $h\rightarrow 0.$
By \eqref{lim003} and \eqref{lim004}, we have that $H\subset \overline{U_0}=\overline{Dv_0(\mathbb{R}^n)}.$ By convexity we also have that the convex set $U_0$ splits, namely, $U_0=H\times \omega$ for some two dimensional convex set $\omega.$ Finally we can apply the argument in \cite{CLWf} to arrive a contradiction.

\subsection{Case (ii)} 
By a translation of coordinates we may assume $x_0=y_0=0.$
Denote by $e_i$, $i=1, \cdots, n$, the standard coordinate directions of $\mathbb{R}^n.$
Up to a rotation of coordinates we may assume $\nu_{\Omega}(0) =e_n$ and $\nu_{\Omega^*}(0)=e_1$. 
Denote by $H_{\mathcal{F}}(0)$ the tangent hyperplane of $\overline{\mathcal{F}}$ at $0$. 
Denote by $H_{\Omega}(0):=\{x_n=0\}$ the tangent hyperplane of $\Omega$ at $0$. 
Let $H':=H_{\mathcal{F}}(0)\cap H_{\Omega}(0),$ we have that $H'$ is an $(n-2)$-dimensional subspace of $\mathbb{R}^n$. (Otherwise, if $H'$ is $(n-1)$ dimensional, it goes back to Case (i).)
Observe that if $e_1\in H',$ then $e_1\cdot  \nu_{\overline{\mathcal{F}}}(0)=0,$ which again can be reduced to Case (i). Hence, in the following we always assume $e_1\notin H'.$
Similar to the proof of \cite[Lemma 5.20]{CLWf}, up to an affine transformation we may assume 
$H'=\text{span}\{e_2, \cdots, e_{n-1}\}.$
Now, we can carry out a blow up argument developed in \cite{CL1} as follows.

Let 
	$$d_e:=\sup\{|y\cdot e|: y\in S_h[v]\cap \Omega^*\}$$ for any unit vector $e.$
Denote $d_{e_i}$ by $d_i$ for short.  Let $p_e\in \overline{S_h[v]\cap \Omega_1^*} $ be the point such that $p_e\cdot e=d_e.$
We have the following estimate.
\begin{lemma}\label{ide1}
$|d_1|\leq  C_{\epsilon}h^{\frac{2}{3}-\epsilon},$ $d_e^2\leq C d_1$ for any unit vector
$e\in \text{span}\{e_2, \cdots, e_n\},$
where $\epsilon$ can be as small as we want.
\end{lemma}
\begin{proof}
Suppose $d_1=p_{e_1}\cdot e_1 \geq h^{\frac{2}{3}-4n\epsilon}.$
Let $q$ be the intersection of the ray $\{p_{e_1}-te_n: t\geq 0\}$ and $\partial \Omega^*,$
we have
 $q_1=d_1\geq  h^{\frac{2}{3}-4n\epsilon}.$
Since $Dv(\mathbb{R}^n)\subset \overline{\Omega}\subset\{x_n\geq 0\},$ we have that $v$ is increasing in $e_n$ direction. Hence $v(q)\leq v(p_{e_1})=h,$ which implies that $q\in \overline{S_h[v]}.$

%After a change of coordinates, we may assume that $q$ is on the $x_1x_2$ plane. By $C^2$ regularity of $\Omega_1^*$ we have that $|q_2|\gtrsim \sqrt{q_1}\gtrsim h^{\frac{1}{3}-2n\epsilon}.$

Denote $e_q:=\frac{q-y_0}{|q-y_0|}.$ Denote by $D$ the planar region in $\text{span}\{e_q, e_1\},$ enclosed by 
$\partial\Omega_1^*\cap \text{span}\{e_q, e_1\}$ and the segment $y_0q.$
By the $C^2$ regularity and uniform convexity of $\partial\Omega_1^*$ we have that $\mathcal{H}^2(D)\geq Cd_1^{\frac{3}{2}}.$
Let $\tilde{e}_2, \cdots ,\tilde{e}_{n-1}$ be an orthonormal basis of the orthogonal complement of $\text{span}\{e_1, e_p\}$ in $\mathbb{R}^n.$ By the tangential $C^{1, 1-\epsilon}$ estimate of $v$ we have that $y_0+h^{\frac{1}{2}+\epsilon}\tilde{e}_i\in S^c_{bh}[v]$, $i=2,\cdots, n-1.$ Let $G$ be the convex envelope of 
$D$ and the points $y_0+h^{\frac{1}{2}+\epsilon}\tilde{e}_i$, $i=2,\cdots, n-1.$ By convexity we have $G\subset S^c_{bh}[v].$
Hence $$h^{(\frac{1}{2}+\epsilon)(n-2)}\mathcal{H}^2(D)\leq |G|\leq |S^c_{bh}[v]|\approx h^{\frac{n}{2}},$$ which
implies that $|d_1|\leq  C_{\epsilon}h^{\frac{2}{3}-\epsilon}.$  

Fix any unit vector
$e\in \text{span}\{e_2, \cdots, e_n\},$
by the uniform convexity of $\partial\Omega_1^*$ we have that $d_1\geq p_e\cdot e_1\geq C(p_e\cdot e)^2,$ which implies $d_e^2\leq C d_1.$ 
\end{proof}

{\bf First blow up.}  Suppose
$S_h^c[v]\sim E$ for some ellipsoid centred at $0.$ Then, $E\cap\{x_1=0\}$ is an $(n-1)$-dimensional ellipsoid with principal directions $\bar e_2, \cdots, \bar e_n.$ Note that $\text{span}\{\bar e_2, \cdots, \bar e_n\}=\text{span}\{e_2, \cdots, e_n\}.$
Now, we can write 
$$E=\left\{x=x_1e_1+\sum_{i=2}^n\bar{x}_i\bar e_i: \frac{x_1^2}{a_1^2}+\sum_{i=2}^n\frac{(\bar x_i-k_ix_1)^2}{a_i^2} \leq 1\right\}.  $$
By Lemma \ref{ide1} and the tangential $C^{1,1-\epsilon}$ estimate of $v$ at $y_0,$  
we have that  
\begin{equation}\label{lim200}
0<a_1<C_\epsilon h^{\frac{2}{3}-\epsilon},\ \text{and}\ C_\epsilon h^{\frac{1}{2}+\epsilon}<a_i< C_\epsilon h^{\frac{1}{3}-\epsilon}\ \text{for}\ i=2, \cdots, n.
\end{equation}
Let $\tilde{d}_e:=\{|x\cdot e|: x\in S_h^c[v]\}.$ By  uniform density and Lemma \ref{ide1} we have 
\begin{equation}\label{wid111}
\tilde d_e^2\leq Ca_1\ \text{for any}\ e\in \text{span}\{e_2, \cdots, e_n\}.
\end{equation}
 It follows that 
\begin{equation}\label{estwid21}
|k_i|\leq C\frac{a_1^{\frac{1}{2}}}{a_1}=Ca_1^{-\frac{1}{2}}.
\end{equation}

Let $T_1, T_2$ be the affine transformations as following.
$$T_1: x=x_1e_1+\sum_{i=2}^n\bar{x}_i\bar e_i \mapsto  z=x_1e_1+  \sum_{i=2}^n(\bar{x}_i-k_ix_1)\bar e_i;$$
$$T_2: z=z_1e_1+\sum_{i=2}^n\bar{z}_i\bar e_i \mapsto y=\frac{z_1}{a_1}e_1+\sum_{i=2}^n\frac{\bar{z}_i}{a_i}\bar e_i.$$
Then $T_h(E)=B_1,$ where $T_h=T_2T_1.$ Hence $T_hS_h^c[v]\sim B_1.$
Let $v_h(x):=\frac{1}{h}v(T_h^{-1}x).$ Then $Dv_h(x)=T_h^*Dv(T_h^{-1}x),$ 
where $T_h^*=\frac{(T_h^t)^{-1}}h=\frac{1}{h}(T_2^t)^{-1}(T_1^t)^{-1}.$
A direct computation shows that 
$$(T^t_1)^{-1}: x=x_1e_1+\sum_{i=2}^n\bar{x}_i\bar e_i \mapsto  z=(x_1+\sum_{i=2}^nk_i\bar x_i)e_1+  \sum_{i=2}^n\bar x_i\bar e_i;$$
$$(T^t_2)^{-1}: z=z_1e_1+\sum_{i=2}^n\bar{z}_i\bar e_i \mapsto y=a_1z_1e_1+\sum_{i=2}^n a_i\bar{z}_i\bar e_i.$$

Let $\mathcal{M}$ be as in Lemma \ref{localgood}. For any given unit vector 
$e\in \text{span}\{e_2,\cdots, e_{n-1}\},$ 
Let $t_h$ be the positive number such that
$|\frac{1}{h}(T_2^t)^{-1}(t_he)|=1,$ we have that
$$t_h\leq \frac{h}{\min_{2\leq i\leq n}a_i}\leq C_\epsilon h^{\frac{1}{2}-\epsilon}.$$
For any $t<h^{-\epsilon}t_h,$ denote $p_t:=te=(0, x^t_2,\cdots, x^t_{n-1},0).$
%We use $\rho(x^t)$ to denote $\rho(0, x^t_2,\cdots, x^t_{n-1})$ for short.
 It is straightforward
to check that 
\begin{equation}\label{lim201}
|T_h^*(h^{-\epsilon}t_he)|\rightarrow \infty, \ \text{as}\ h\rightarrow 0.
\end{equation}
By Lemma \ref{localgood}, we can find $q_t=p_t+x^t_1e_1+x^t_ne_n\in \mathcal{M}\cap B_{r_0}(0)\subset \overline{U},$ with 
\begin{equation}\label{lim202}
|x^t_1|, \ |x^t_n|\leq Ct^2\leq Ch^{1-4\epsilon}.
\end{equation}
Let $z_i=e_n\cdot \bar e_i,$ we have that $e_n=\sum_{i=2}^nz_i\bar e_i.$
Then 
$$T_h^*(x^t_1e_1+x^t_ne_n)=\frac{1}{h}x^t_n\left(a_1(\sum_{i=2}^nk_iz_i)e_1+\sum_{i=2}^na_iz_i\bar e_i\right)+\frac{1}{h}a_1x^t_1e_1.$$

Since $|z_i|\leq 1$ for $i=2,\cdots, n-1,$
by \eqref{lim200}, \eqref{estwid21} and \eqref{lim202}, a straightforward computation shows that
$$|T_h^*(x^t_1e_1+x^t_ne_n)|\leq Ch^{-4\epsilon}(h^{\frac{1}{3}-\frac{\epsilon}{2}}+h^{\frac{1}{3}-\epsilon})+h^{-4\epsilon}h^{\frac{2}{3}-\epsilon}\leq Ch^{\frac{1}{3}-5\epsilon}.$$
Hence,
\begin{equation}\label{lim203}
|T^*_h(q_t-p_t)|\rightarrow 0 \  \text{ as }\ h\rightarrow 0,
\end{equation}
provided $\epsilon$ is sufficiently small.

Let $v_h$ be as in \eqref{novh1}. Up to a subsequence we may assume that
$T_h(\Omega^*_h)$ converges to a convex set $\Omega^*_0$ locally uniformly as $h\rightarrow 0,$ and that $v_h$ converges to a convex function 
$v_0$ locally uniformly as $h\rightarrow 0,$ 
and $v_0$ satisfies \eqref{maeq111}.
Let $U_0$ be the interior of the convex set $\overline{\partial v_0(\mathbb{R}^n)}.$
Up to a subsequence we may assume $T_h^*(H')$ converges to $H$ for some $n-2$ dimensional subspace of $\mathbb{R}^n$ as $h\rightarrow 0.$ 
% By the definition of $T_h^*$, it is also straightforward to check that $H\subset \text{span}\{e_2, \cdots, e_n\}.$
Note that $e_h:=\frac{T_he_n}{|T_he_n|}$ is the unit inner normal of $T_h^*(\Omega).$ Up to a subsequence we may assume that $e_h$ converges to a unit vector $e_0.$ By the definition of $T_h$ we have that $e_h\cdot e_1=0.$ Passing to limit we have that $e_0\cdot e_1=0.$
Since $T_h^*(H')$ is a subspace of the tangent space of $\partial T_h^*(\Omega)$ at $0,$ we have that $e_h$ is orthogonal to $T_h^*(H'),$ hence passing to limit we have that $e_0$ is orthogonal to $H.$ Now, by a rotation of $e_2,\cdots, e_n$ coordinates we may assume
$e_0=e_n$, $H=\text{span}\{e_2,\cdots, e_{n-1}\}.$
By \eqref{lim201} and \eqref{lim203} we have that $U_0$ splits, namely, $U_0=H\times \omega$ for some two dimensional convex set $\omega$. Moreover $U_0\subset \{x_n\geq 0\}.$ 
Similar to the proof of \cite[Lemma 3.5]{CL1} we also have that $\Omega^*_0$ is a smooth convex domain satisfying $\Omega^*_0=\{x_1\geq P(x_2, \cdots, x_n\}$ for some non-negative homogeneous quadratic polynomial satisfying $P(0)=0, DP(0)=0.$ 

Finally, we can follow the argument  in \cite{CL1} to perform a second blow up and then use the argument in \cite{CLW1} to deduce a contradiction.

\section{$C^{1, 1-\epsilon}$ estimate}
In this section we assume $\Omega, \Omega^*, f, g$ satisfy the same conditions as those in Proposition \ref{obli}. 
We will establish the $C^{1, 1-\epsilon}$ estimate of $u.$ For any $y_0\in \partial\Omega^*,$ denote
$x_0=Dv(y_0)\in \partial U.$ In the following we will establish a pointwise $C^{1, 1-\epsilon}$ estimate of $v$ at $x_0.$  After a translation of coordinates and subtracting an affine function to $v$, we may assume that $x_0=y_0=0,$ and that $v(0)=0$, $v\geq 0.$
\begin{theorem}\label{tc111}
 There exists a small constant $r>0$ independent of the location of $y_0=0$ on $\partial\Omega^*,$ such that 
$$0\leq v(y)\leq C_\epsilon |y|^{2-\epsilon}$$
for any $y\in B_{r}(0).$
\end{theorem}

We will show such estimate in the worst scenario when $x_0\in \overline{\mathcal{F}}\cap \partial\Omega.$ Note that near such point, $\partial U$ is only known to be Lipschitz, and it is also not known whether $\partial U$ is locally convex. 
By Proposition \ref{obli}, we can find an affine transform $A$ such that
$(A^t)^{-1}\nu_{\Omega^*}(0)$ (the inner normal direction of $A\Omega^*$ at $0$) is parallel to
$A\nu_{\Omega}(0)$ (the inner normal direction of $(A^t)^{-1}\Omega$ at $0$).
Hence up to an affine transformation we may assume $\nu_{\Omega^*}(0)=\nu_{\Omega}(0)=e_n.$
Moreover, by Proposition \ref{obli} we also have that there exists constants $K, r>0$ such that
\begin{equation}\label{doposi}
\mathcal{C}_{K, r}:=\{x_n\geq K|x'|\}\cap B_r(0)\subset U.
\end{equation}

Let $z_h=c_he_n$ be the intersection of positive $x_n$ axis and $\partial S_h^c[v].$ To prove Theorem \ref{tc111} we only need to establish the following lemma.
\begin{lemma}\label{ltc111} For any $\eta>0$ small, there exists a constant $C_\eta,$ such that
$c_h> C_\eta h^{\frac{1}{2}+\eta}.$
\end{lemma}
\begin{proof}
Suppose to the contrary that for some $\eta>0$ small one cannot find such $C_\eta,$ namely, for 
any given $c_0>0,$ there exists $h>0$ such that $c_h\leq c_0 h^{\frac{1}{2}+\eta}.$ Let 
$h_0:=\max\{h: c_h\leq c_0 h^{\frac{1}{2}+\eta}.$ By taking $c_0\rightarrow 0,$ it is straightforward to check that $h_0\rightarrow 0.$ Hence $c_{h_0}=c_0 h_0^{\frac{1}{2}+\eta}.$

Let $M$ be a large constant to be determined later.
There exists a constant $C$ depending only on $n,$ such that  $S_{Mh_0}^c[v]\subset CMS^c_{h_0}[v].$
Suppose 
$S_{Mh_0}^c[v]\sim E$ for some ellipsoid centred at $0.$ Then, $E\cap\{x_n=0\}$ is an $(n-1)$-dimensional ellipsoid. By a rotation of $x_1, \cdots, x_{n-1}$ coordinates we may assume that the principal directions of $E\cap\{x_n=0\}$ are $e_1, \cdots, e_{n-1}.$ We can write
$$E=\left\{\sum_{i=1}^nx_i e_i: \sum_{i=1}^{n-1}\frac{( x_i-k_ix_1)^2}{a_i^2} + \frac{x_n^2}{a_n^2} \leq 1\right\}.  $$
By tangential $C^{1, 1-\epsilon}$ of $v$ we have that
\begin{equation}\label{contra11}
a_i\geq C_\epsilon h_0^{\frac{1}{2}+\epsilon}, \ \text{for}\ i=1,\cdots, n-1.
\end{equation}
Let $p_{h_0}=t_{h_0}e_n$ be the intersection of the positive $x_n$ axis and $\partial E.$
By the contradiction assumption  we also have
\begin{equation}\label{contra22}
t_{h_0}\geq C_nc_0(Mh_0)^{\frac{1}{2}+\eta}.
\end{equation}

Let $E^*$ be the dual ellipsoid of $E,$ namely, 
\begin{equation}\label{contra33}
E^*:=\{y: y\cdot x\leq h\ \text{for any}\ x\in E\}.
\end{equation}
Then, $E^*\cap\{x_n=0\}$ is an $(n-1)$-dimensional ellipsoid with principal directions $\bar e_1, \cdots, \bar e_{n-1}.$ Note that $\text{span}\{\bar e_1, \cdots, \bar e_{n-1}\}=\text{span}\{e_1, \cdots, e_{n-1}\}.$ We can now write
$$E^*=\left\{y=y_ne_n+\sum_{i=1}^{n-1}\bar{y}_i\bar e_i: \sum_{i=1}^{n-1}\frac{(\bar y_i-k'_iy_n)^2}{b_i^2}+\frac{y_n^2}{b_n^2}\leq 1\right\}.  $$
By \eqref{contra22} and \eqref{contra33} we have that
\begin{equation}\label{contra44}
b_n\geq \frac{C}{c_0}h_0^{\frac{1}{2}-\eta}.
\end{equation}
By \eqref{contra11} we have that for any $e\in \text{span}\{e_1, \cdots, e_{n-1}\},$
\begin{equation}\label{contra55}
\sup\{y\cdot e: y\in E^*\}\leq \frac{1}{C_\epsilon}h_0^{\frac{1}{2}-\epsilon} \quad \forall\,\epsilon>0.
\end{equation}
Hence the following estimates holds:
\begin{equation}\label{contra66}
|k_i'|\leq \frac{\frac{1}{C_\epsilon}h_0^{\frac{1}{2}-\epsilon}}{b_n}\leq \frac{\frac{1}{C_\epsilon}h_0^{\frac{1}{2}-\epsilon}}{\frac{C}{c_0}h_0^{\frac{1}{2}-\eta}}\leq Ch_0^{\eta-\epsilon}
\rightarrow 0\ \text{as}\ h_0\rightarrow 0,
\end{equation}
provided $\epsilon<\eta.$
By \eqref{contra55} we also have
\begin{equation*}
|b_i|\leq  \frac{1}{C_\epsilon}h_0^{\frac{1}{2}-\epsilon}.
\end{equation*}
Hence 
\begin{equation}\label{contra77}
\frac{b_n}{b_i}\geq C_\epsilon h_0^{\epsilon-\eta}  \rightarrow \infty \ \text{as}\ h_0\rightarrow 0\ \text{for}\ i=1,\cdots, n-1.
\end{equation}

Let $A_{h_0}$ be an affine transformation such that $A_{h_0}E=B_1(0),$ then $A_{h_0}^*(E^*)=B_1(0),$ where
$A_{h_0}^*:=\frac{1}{h_0}(A_{h_0}^t)^{-1}.$ 
%Note that we can take $A_{h_0}$ to be positive definite.
Let $T_{h_0}$ be the affine transform
$$T_{h_0}: x=x_1e_1+\sum_{i=2}^n\bar{x}_i\bar e_i \mapsto  z=\frac{x_n}{b_n}e_n+  \sum_{i=1}^{n-1}\frac{\bar{x}_i-k'_ix_n}{b_i}\bar e_i.$$
Then, $T_{h_0}E^*=B_1(0).$ Hence $A^*_{h_0}T_{h_0}^{-1}$ is an affine transform satisfying 
$A^*_{h_0}T_{h_0}^{-1}(B_1(0))=B_1(0),$ which implies that $A^*_{h_0}T_{h_0}^{-1}$ is a linear isometry from $\mathbb{R}^n$ to $\mathbb{R}^n.$
By \eqref{contra66} and \eqref{contra77}, a direct computation shows that $T_{h_0}\mathcal{C}_{K, r}$ converges to the upper half space $\{x_n\geq 0\}$ locally uniformly in Hausdorff distance as $h\rightarrow 0.$
Since $A_{h_0}T_{h_0}^{-1}$ is a linear isometry, we have that 
\begin{equation}\label{contra88}
A^*_{h_0}\mathcal{C}_{K, r}=A^*_{h_0}T_{h_0}^{-1}T_{h_0}\mathcal{C}_{K, r}\rightarrow \{x\in \mathbb{R}^n: x\cdot e_0\geq 0\}
\end{equation}
 locally uniformly in Hausdorff distance for some unit vector $e_0,$ as $h_0\rightarrow 0.$

Similar to the proof of \cite[Lemma 3.5]{CL1},
up to a subsequence we may assume $A_{h_0}(\Omega^*)$ locally uniformly converges to a smooth convex set $\Omega_0^*$ as $h_0\rightarrow 0.$
Let
	\begin{equation}\label{novvh1}
		v_{h_0}(y):=\frac{1}{Mh_0}v(A_{h_0}^{-1}y).
	\end{equation}
Then, $v_h$ is locally uniformly bounded in $\R^n$ as $h\to0$. 
Hence by passing to a subsequence, $v_{h_0}\rightarrow v_0$, locally uniformly,
and $v_0$ satisfies 
\begin{equation}\label{maeq222}
\det\, D^2 v_0=c_0\chi_{_{\Omega^*_0}}\quad\text{in } \R^n
\end{equation}
for a constant $c_0>0$.
Since $v_0$ is a convex function defined on entire $\mathbb{R}^n,$ we have that
$U_0,$ the interior of $\partial v_0 (\mathbb{R}^n)$ is a convex set.  By \eqref{contra88} we have that 
$U_0=\{x\in \mathbb{R}^n: x\cdot e_0> 0\},$ in particular $\partial U_0$ is smooth.
Since both $\Omega^*_0$ and $U_0$ are smooth convex, $v_0$ is a convex solution to \eqref{maeq222} satisfying $v_0(0)=0, v_0\geq 0,$ hence by the regularity theory of \cite{CLW1} we have that $0\leq v_0(x)\leq C|x|^2$ for any $x\in \Omega_0^*.$ Hence 
$B_{ch^{\frac{1}{2}}}(0)\cap \Omega^*_0\subset S_{h}[v_0].$ Since 
$S_{\frac{1}{C}h}[v_0]\subset S_{h}^c[v_0]\cap \Omega^*_0$ and $|S_{h}[v_0]|\approx h^{\frac{n}{2}},$ we have that  $S_{h}^c[v_0]\sim B_{h^{\frac{1}{2}}}(0).$
Since $v_0(y)\leq Ch$ for any $y\in S_{h}^c[v_0],$ it follows that
\begin{equation}\label{contra99}
v_0(y)\leq C|y|^2
\end{equation}
for any $y\in B_1(0).$ 

There exists a constant $C_1$ depending only on $n,$ such that $-C_1c_{h_0}e_n\notin S^c_{h_0}[v].$
Let $h=C_1c_{h_0}=C_1c_0h_0^{\frac{1}{2}+\eta}.$
Hence
 \begin{equation}\label{c11contra1}
 \Delta_hv:=v(he_n)+v(-he_n)-2v(0)\geq 2h_0.
 \end{equation}
Denote $e_{h_0}=A_{h_0}(t_{h_0}e_n).$ 
Then 
\begin{eqnarray*}
\Delta_hv&=&Mh_0\left(v_{h_0}(t_{h_0}^{-1}he_{h_0})+v_{h_0}(-t_{h_0}^{-1}he_{h_0})-2v_{h_0}(0)\right)\\
&\leq& Mh_0\left(Ct_{h_0}^{-2}h^2+4\|v_{h_0}-v_0\|_{L^\infty(B_1(0))}\right)\\
&\leq& h_0\left(CM\frac{C_1^2c_0^2h_0^{1+2\eta}}{C_n^2c_0^2M^{1+2\eta}h_0^{1+2\eta}}+4\|v_{h_0}-v_0\|_{L^\infty(B_1(0))} \right)\\
&=&h_0\left(CC_1^2C_n^{-2}M^{-2\eta}+4\|v_{h_0}-v_0\|_{L^\infty(B_1(0))}\right).
\end{eqnarray*}
By first taking $M$ sufficiently large we can have $CC_1^2C_n^{-2}M^{-2\eta}\leq \frac{1}{4}.$ Then
by taking $h_0$ sufficiently small and  we can have $4\|v_{h_0}-v_0\|_{L^\infty(B_1(0))}\leq \frac{1}{4}.$ 
Now $\Delta_hv\leq h_0(\frac{1}{4}+\frac{1}{4})=\frac{1}{2}h_0,$ contradicting to \eqref{c11contra1}.

Therefore, we can always find the desired $C_\eta,$ such that $c_h> C_\eta h^{\frac{1}{2}+\eta}.$
\end{proof}

\section{$W^{2,p}$ estimate}
By Theorem \eqref{tc111}, we have that $B_{C_\epsilon h^{\frac{1}{2}+\epsilon}}(0)\cap \Omega^*\subset S_h[v].$ 
Then, since  $|S_h[v]|\approx h^{\frac{n}{2}},$ we have that 
$$B_{C_\epsilon h^{\frac{1}{2}+\epsilon}}(0)\cap \Omega^*\subset S^c_h[v]\subset B_{C_\epsilon h^{\frac{1}{2}-\epsilon}}(0).$$
This implies $v(y)\geq C_\epsilon |y|^{2+\epsilon}.$
Since $u=v^*,$ we have that $u(x)\leq C_\epsilon |x|^{2-\epsilon}$ for any $x\in \overline{U}.$
 Hence, $u\in C^{1,1-\epsilon}(\overline{U}).$ 
 The proof of Theorem \ref{main2} can then be completed following the proof of \cite[Theorem 1.2]{CLW1}.

\section{Counter example for $C^{1, 1}$ regularity}

The $W^{2, p}$ regularity of $u$ in $\overline{U}$ is sharp in the sense that we can construct examples showing that even assuming further that the densities are smooth, the solution $u$ is not global $C^{1,1}$ in $U.$ The examples can be constructed as follows.
Let $p_0=(2, 0)\in \mathbb{R}^2.$  Let $\Omega=B_{1}(0), \Omega^*=B_{\frac{1}{100}}(p_0)$ and 
$f=\chi_{\Omega}, g=\chi_{\Omega^*}.$ Let $u$ be the solution to the obstacle problem \eqref{mao}.
First, we claim that 
\begin{equation}\label{example111}
U\subset B_1(0)\cap \{x_1> 0\}.
\end{equation}
  Suppose not, then there exists $p\in U \cap \{x_1\leq 0\}.$ Denote $q_0:=(1, 0).$  Since $\int_{B_{\frac{1}{10}}(q_0)\cap \Omega}f>\int_{\Omega^*}g$
and $$\text{dist}(p, \Omega^*)>\sup\{|x-y|: x\in B_{\frac{1}{10}}(q_0)\cap \Omega, y\in \Omega^*\},$$ 
we can construct a cheaper transport plan by
 replacing the mass in $U$ near $p$ by the mass in $\left(B_{\frac{1}{10}}(q_0)\cap \Omega\right)\backslash U,$ which is a contradiction. Let $x_0$ be a point in $\overline{\mathcal{F}}\cap \partial\Omega.$ 
 By \eqref{normalformular} we have that $\nu_{\overline{\mathcal{F}}}(x_0)=\frac{Du(x_0)-x_0}{|Du(x_0)-x_0|}.$ Then by Theorem \ref{main2} and \eqref{example111}, it is straightforward to verify that 
 $-1<\nu_{\overline{\mathcal{F}}}(x_0)\cdot \nu_{\Omega}(x_0)<1.$

Now we adapt an argument of Savin and Yu \cite{SY, SY1} to show that $u\notin C^{1, 1}(\overline{U}).$
Suppose to the contrary $u\in C^{1, 1}(\overline{U}).$ By a translation of coordinates and subtracting 
an affine function, we may assume $x_0=0, Du(x_0)=0, u\geq 0.$ Up to an affine transform we may assume $\nu_{\overline{\mathcal{F}}}(0)=\nu_{\Omega^*}(0)=e_2.$ Since $\partial\Omega,  \overline{\mathcal{F}}$ are all at least $C^1$ regular, then $\lambda  U$ converges to a cone $\mathcal{C}:=\{tz: t>0, z=(\cos\theta, \sin\theta), \theta\in (0, \theta_0)\}$ as $\lambda\rightarrow \infty,$ where $0<\theta_0<\pi.$ By Proposition \ref{obli} we actually have $\frac{\pi}{2}<\theta_0<\pi.$
 Since $\partial\Omega^*$ is $C^2$ we have that $\lambda\Omega^*$ converges to $\{x_2>0\}$ as $\lambda\rightarrow\infty.$
Now, let $u_h(x):=\frac{1}{h}u(h^{\frac{1}{2}}x).$ Then up to a subsequence $u_h$  converges to a  convex function $u_0$ locally uniformly as $h\rightarrow 0,$ with $u_0\in C^{1,1}(\mathcal{C})$ and $Du_0(\mathcal{C} )=\{x_2>0\}.$ Let $v_0(y):=\sup_{x\in \mathcal{C}}x\cdot y-u_0(y),$ then
$v_0\in C^{1,1}(\{x_2\geq 0\}),$ $\det D^2v_0=1$ in $\{x_2>0\}$ and $Dv_0(\{x_2>0\})=\mathcal{C}.$

Fix any unit vector $e,$ let $p_k\in \{x_2>0\}$ be a sequence such that $\partial_{_{ee}}v_0(p_k)$ converges to 
$\sup_{x\in \{x_2>0\}}\partial_{_{ee}}v_{0}(x).$ Let $v_{0k}(x)=\frac{1}{|p_k|^2}v_0(|p_k| x).$ Note that quadratic rescaling preserves second derivatives.
  By compactness, we may also assume $\frac{p_k}{|p_k|}$ converges to a point $p_\infty\in \{x_2\geq 0\}$  as $k\rightarrow \infty.$ Note that $|p_\infty|=1.$
Up to a subsequence we may assume $v_{0k}$ converges to a convex function $\tilde{v}_0$ locally uniformly as $k\rightarrow \infty.$ Now $\tilde{v}_0$ satisfies 
$\tilde v_0\in C^{1,1}(\{x_2\geq 0\}),$ $\det D^2\tilde v_0=1$ in $\{x_2>0\}$ and $D\tilde v_0(\{x_2>0\})=\mathcal{C}.$ Moreover, the maximum of $\partial_{_{ee}}\tilde{v}_0$ is attained at 
$p_\infty\in \{x_2\geq 0\}.$ If $p_\infty\in \{x_2>0\},$ then by maximum principle we have $\partial_{_{ee}}\tilde{v}_0$ is constant in $\{x_2> 0\}.$ If $p_\infty \in \{x_2=0\},$ since $|p_\infty|=1,$ we have 
$p_\infty=(1, 0)$ or $p_\infty=(-1, 0).$ In either case, up to an affine transform we may assume 
$D\tilde{v}_0\cdot e_2=0$ in $B_{\frac{1}{2}}(p_\infty)\cap \{x_2=0\}.$ Then we can extend $\tilde v_0$ to the entire $B_{\frac{1}{2}}(p_\infty)$ by reflection, namely let
$\tilde{v}_0(x_1, x_2)=\tilde{v}_0(x_1, -x_2)$ whenever $x_2<0.$ Then $\tilde v_0$ is a strictly convex function in $B_{\frac{1}{2}}(p_\infty)$ satisfies $\det D^2 \tilde v_0=1$ in  $B_{\frac{1}{2}}(p_\infty).$  Now $\partial{_{ee}}\tilde{v}_0$ attains its maximum in an interior point $p_\infty,$ by maximum principle again we have that $\partial_{_{ee}}\tilde{v}_0$ is constant in $\{x_2>0\}.$ 

Therefore, $\tilde{v}_0$ is a quadratic polynomial and hence $D\tilde{v}_0$ maps half space to half space, which contradicts to $D\tilde v_0(\{x_2>0\})=\mathcal{C}.$  Hence, the contradiction assumption fails, namely $u\notin C^{1, 1}(\overline{U}).$

%by \eqref{equi0} we have $B_{C_\epsilon h^{\frac{1}{2}+\epsilon}}(0)\subset S^c_h[v].$ 
%Combining this inclusion with $|S^c_h[v]|\approx h^{\frac{n}{2}},$ we have that 
%$B_{C_\epsilon h^{\frac{1}{2}+\epsilon}}(0)\subset S^c_h[v]\subset B_{C_\epsilon h^{\frac{1}{2}+\epsilon}}(0)

\vskip10pt

\bibliographystyle{amsplain}

\begin{thebibliography}{12}

%\bibitem{AC} E. Andriyanova and S. Chen,
 %                    Boundary $C^{1,\alpha}$ regularity of potential functions in optimal transportation with quadratic cost.
  %                   \emph{Analysis and PDE.}, 9(2016), 1483--1496.

%\bibitem{Ca77} L. A. Caffarelli,
%		The regularity of free boundaries in higher dimensions.
%		\emph{Acta Math.}, 139 (1977), 155--184.

\bibitem{C91}  L. A. Caffarelli,
                    Some regularity properties of solutions of Monge Amp\`ere equation. 
                    \emph{Comm. Pure Appl. Math.}, 44 (1991), 965--969.

\bibitem{C92} L. A. Caffarelli,
                      The regularity of mappings with a convex potential. 
                      \emph{J. Amer. Math. Soc.}, 5 (1992), 99--104.
 
\bibitem{C92b} L. A. Caffarelli,
                       Boundary regularity of maps with convex potentials. 
                       \emph{Comm. Pure Appl. Math.}, 45 (1992), 1141--1151.

\bibitem{C96}  L. A. Caffarelli,
                       Boundary regularity of maps with convex potentials II.
                       \emph{Ann. of Math.}, 144 (1996), 453--496.
                       
%\bibitem{Ca98} L.A. Caffarelli, 
   %              The obstacle problem revisited.
      %             \textit{J. Fourier Anal. Appl.,} 4 (1998), 383--402. 

 \bibitem{CM}  L. A. Caffarelli and R. J. McCann, 
                       Free boundaries in optimal transport and Monge-Amp\`{e}re obstacle problems.
                       \textit{Ann. of Math.}, 171 (2010), 673--730.

%\bibitem{CL1}  S. Chen, J. Liu,
         %              Regularity of free boundaries in optimal transportation. 
         %              arXiv:1911.10790.  
         %              {\Small\color{blue} (How to mention this paper in the context?)}
         
       \bibitem{CL1}  S. Chen and J. Liu,
                       Regularity of singular set in optimal transportation.
                        \emph{arXiv:2210.13841}

  

\bibitem{CLW1}  S. Chen; J. Liu and X.-J. Wang, 
                       Global regularity for the Monge-Amp\`ere equation with natural boundary condition.
                        \emph{Ann. of Math.}, 194 (2021), 745--793.

\bibitem{CLW2} S. Chen; J. Liu and X.-J. Wang, 
                       Boundary regularity for the second boundary-value problem of Monge-Amp\`ere equations in dimension two, 
                       arXiv:1806.09482.
                       
                       
\bibitem{CLWf} S. Chen; J. Liu and X.-J. Wang, 
                     $C^{2,\alpha}$ regularity of free boundary in optimal transport.
                     Accepted by CPAM.

\bibitem{CW1} S. Chen and X.-J. Wang, 
                        Strict convexity and $C^{1,\alpha}$ regularity of potential functions in optimal transportation under condition A3w. 
                     \textit{J. Differential Equations}, 260 (2016),1954--1974,

\bibitem{D91} Ph. Delano\"e,
	             Classical solvability in dimension two of the second boundary value problem associated with the Monge-Amp\`ere operator.
	             \textit{Ann. Inst. Henri Poincar\'e, Analyse Non Lin\'eaire}, 8 (1991), 443--457.

\bibitem{AFi2} A. Figalli, 
                       A note on the regularity of the free boundaries in the optimal partial transport problem.
                       \textit{Rend. Circ. Mat. Palermo}, 58 (2009),  283-286.
 	    
 \bibitem{AFi}  A. Figalli, 
                       The optimal partial transport problem. 
                        \textit{Arch. Ration. Mech. Anal.}, 195 (2010), 533--560.
  
\bibitem{Fig3} A. Figalli,
                       \textit{The Monge-Amp\`ere equation and its applications},
                        {Zurich Lectures in Advanced Mathematics,}
                        \textit{European Mathematical Society}, 2017.

 \bibitem{GT} D. Gilbarg and N. S. Trudinger. 
                      \textit{Elliptic partial differential equations of second order}. 
                      Springer-Verlag, Berlin, 2001.

 \bibitem{I}     E. Indrei, 
                      Free boundary regularity in the optimal partial transport problem. 
                      \textit{J. Funct. Anal.}, 264 (2013), 2497--2528.   

\bibitem{JW}   H. Y. Jian  and X.-J. Wang, 
                       Continuity estimates for the Monge-Amp\`ere equation,
                      \textit{SIAM J. Math. Anal.}, 39 (2007), 608--626.

\bibitem{KM1} J. Kitagawa and R. McCann,
                        Free discontinuties in optimal transport.
                         \emph{Arch. Rational Mech. Anal.},  232 (2019), 1505--1541
                       %   {\Small\color{blue} (How to mention this paper in the context?)}
                         
%% \bibitem{LTU}   P.L. Lions, N.S. Trudinger, and J.I.E. Urbas, 
%%                        The Neumann problem for equations of Monge-Amp\`ere type. 
%%                          \textit{Comm. Pure Appl. Math.,} 39 (1986),  539--563. 
                          

%\bibitem{LW}  J. Liu and  X.-J. Wang, 
  %                    Interior a priori estimates for the Monge-Ampère equation. 
  %                    \textit{Surveys in differential geometry} vol.19,  151--177,  Int. Press,  2015. 
  %                    {\Small\color{blue} (Do we need to cite this paper?)}

%\bibitem{MTW} X.-N. Ma; N. S. Trudinger and X.-J. Wang,
%	                Regularity of potential functions of the optimal transportation problem.
%	                \emph{Arch. Ration. Mech. Anal.}, 177 (2005), 151--183.
%	                 {\Small\color{blue} (Do we need to cite this paper?)}

%\bibitem{Sa}  O. Savin, 
 %        	   Pointwise $C^{2,\alpha}$ estimates at the boundary for the Monge-Amp\`ere equation,
    %     	   \textit{J. Amer. Math. Soc.}, 26 (2013), 63--99. {\Small\color{blue} (Do we need to cite this paper?)}

\bibitem{SY}  O. Savin and H. Yu,
                     Regularity of optimal transport between planar convex domains.
   	              {\it Duke Math. J.}, 169 (2020), 1305--1327.  
	            
 \bibitem{SY1}  O. Savin and H. Yu,
                      Online talk.
                     
\bibitem{U1} J. Urbas,
                    On the second boundary value problem of Monge-Amp\`ere type.
                    \textit{J. Reine Angew. Math.},  487 (1997), 115--124.

\bibitem{U2} J. Urbas,
                     Oblique boundary value problems for equations of Monge-Amp\`ere type. 
                       \textit{Calc. Var. PDEs,} 7 (1998), 19--39. 

\bibitem{V1} C. Villani,
                   \textit{Topics in optimal transportation},  Grad. Stud. Math. 58,
                    \textit{Amer. Math. Soc.}, 2003.

\bibitem{V2} C. Villani,
                     \textit{Optimal transport, Old and new}.
                     Springer, Berlin, 2006.
                                         
                     \vskip10pt
                      

\end{thebibliography}

\end{document}